\documentclass[11pt]{article}
\usepackage{amsmath,amssymb,amsbsy,amsfonts,amsthm,latexsym,
            amsopn,amstext,amsxtra,euscript,amscd,amsthm}

\usepackage[usenames, dvipsnames]{color}
\pdfoutput=1
\usepackage{amssymb}
\usepackage{url}
\usepackage{longtable}
\usepackage{xspace}
\usepackage{multirow}
\usepackage{pstricks}
\usepackage[english]{babel}
\usepackage[dvips]{graphicx}
\usepackage{graphicx}
\usepackage{pb-diagram}
\usepackage{picinpar}
\usepackage[latin1]{inputenc}
\usepackage{layout}

\newtheorem{thm}{Theorem}
\newtheorem{lem}[thm]{Lemma}
\newtheorem{prop}[thm]{Proposition}
\newtheorem{cor}[thm]{Corollary}
\newtheorem{remark}[thm]{Remark}

\newcommand \Q{\mathbb{Q}}

\newcommand \Z {\mathbb{Z}}
\newcommand \C {\mathbb{C}}
\newcommand \N {\mathbb{N}}

\begin{document}

\title{Irreducible characters with bounded root Artin conductor }
\author{{Amalia~Pizarro-Madariaga} \\
{Instituto de Matem\'aticas} \\
{Universidad de Valpara\'{\i}so, Chile}\\
{\tt amalia.pizarro@uv.cl}
}

\date{}

\pagenumbering{arabic}

 \maketitle

 \begin{abstract}
In this work, we prove that the best possible lower bound for the Artin conductor is exponential in the degree.

 \end{abstract}

\section{Introduction}

Let $K$ be an algebraic number field  such that $K/{\Q}$ is Galois and let $\chi$ be the character of a linear representation of Gal($K/{\Q}$). We denote by $f_{\chi}$ the Artin conductor of $\chi$.  In \cite{O2}, Odlyzko found lower bounds for  $f_\chi$ by applying analytic methods to the Artin $L$-function. We have improved Odlyzko's lower bounds in  \cite{A1} by using explicit formulas for Artin $L$-functions.  In particular, if $\chi$ is an irreducible character of Gal($K/{\Q}$) and by assuming that $\chi\overline{\chi}$ satisfies the Artin conjecture, we obtained

\begin{equation*}
f_{\chi}^{1/\chi(1)}\geq
4.73(1.648)^{\frac{(a_{\chi}-b_{\chi})^2}{\chi(1)^2}}e^{ -(13.34/\chi(1))^2},
\end{equation*}
where $a_{\chi}$ and $b_{\chi}$ are nonnegative integers giving the $\Gamma$-factors of the completed
Artin $L$-function. Namely, $a_{\chi}+b_{\chi}=\chi(1)$ and $a_{\chi}-b_{\chi}=\chi(\sigma)$, with $\sigma\in$ Gal($K/{\Q}$) the complex conjugation. This bound is even better when we assume that $L(s,\chi\overline{\chi})$ satisfies the Generalized Riemann Hypothesis.  We have to point out that, throughout this article, no additional hypothesis are needed.

A natural question now is how far from being optimal these bounds are. This problem has been studied for the discriminant of a number field. If $n_0=r_1+2r_2$, let $d_n$ be the minimal discriminant of the field $K$ with degree $n$ such that $n$ is a multiple of $n_0$ and $r_1(K)$ and $r_2(K)$ are in the same ratio as $r_1,r_2$.
Let $\displaystyle\alpha(r_1,r_2)=\textrm{lim inf}_{n\to\infty}d_n^{1/n}$. Martinet considered number fields with infinite 2-class field towers and  proved that \cite {Ma1}
\begin{equation*}
\alpha(0,1) < 93\,\,\textrm{and}\,\,\alpha(1,0) < 1059.
\end{equation*}

\noindent In this work, we follow this idea and consider a number field $K$ with infinite $p$-class field tower for some prime $p$. Under some technical conditions on $K$, we find an upper bound (depending only on $K$) for the root Artin conductor of the irreducible characters of Gal($K_n/{\Q}$) (given by $f_{\chi}^{1/\chi(1)}$), where $K_n$ is the Hilbert $p$-class field of $K_{n-1}$ with $K_0=K$.

This work is organized as follow. In \S2, we propose a technique obtained from the Clifford's theory which is useful to classify the irreducible characters of Gal($K_n/{\Q}$) in terms of a certain normal subgroup. This characterization is convenient in order to obtain upper bounds for root Artin conductors. In \S3, we conclude that there exists an infinite sequence $\{\chi_n\}$ of irreducible Artin characters with $\chi_n(1)\to\infty$ and such that $f_{\chi_n}^{1/\chi_n(1)}\leq C$, where  $C>0$ is an effective computable constant. In \S4, we apply the results obtained in \S2 and \S3 to the number field $K=\Q(\zeta_{11}+\zeta_{11}^{-1},\sqrt{2},\sqrt{-23})$. This field was found by Martinet in \cite {Ma1} and  has infinite 2-class field tower and lowest known discriminant. In particular, we prove that
for each $n\geq 1$ it is possible to find an irreducible character of Gal($K_n/{\Q}$) with large degree and

$$f_{\chi}^{1/\chi(1)}\leq C,\,\,\textrm{where}\,\,C \leq 11^4\cdot 2^{15}\cdot 23.$$

\section{Irreducible characters of large degree}

In this section, we develop a technique to classify the irreducible characters of groups with a normal subgroup of prime index. Also, by using a result from \cite{I1},  we obtain conditions that ensure the existence of irreducible characters of large degree. We believe these results are of independent interest.

Let us consider a finite group $G$ and a normal subgroup $H$ of $G$. We denote the set of irreducible characters of $G$ by $Irr(G)$. If $\chi$ and $\theta$ are characters of $G$ and $H$ respectively, we denote the restriction of $\chi$  to $H$ by $Res_H^G \chi$ and the induced character of $\theta$ to $G$ by $Ind_H^G\theta$.
If $\theta\in Irr(H)$,  we define the conjugate character to $\theta$ in $G$ by $\theta^{g}:H\rightarrow{\C}$, where $\theta^{g}(h)=\theta(ghg^{-1})$. The inertia group of $\theta$ in $G$  is given by
\begin{equation*}
I_G(\theta)=\{g\in G:\theta^{g}=\theta\}.
\end{equation*}

\noindent G acts on $Irr(H)$ by conjugation and $I_G(\theta)$ is the stabilizer of $\theta$ under this action.
The next result of Clifford will be the main argument allowing us to give a classification of the irreducible characters of $G$.

\begin{thm}\label{clif} \textup {(Clifford, \cite[p.~253]{Hup})}
 Let $H$ be a normal subgroup of $G$ and $\theta\in Irr(H),\chi\in Irr(G)$ such that $\theta$ is an irreducible constituent of $Res_{H}^G \chi$, with $\langle Res_{H}^G\chi,\theta\rangle=e>0$. Suppose that $\theta=\theta^{g_1},\theta^{g_2},\ldots,\theta^{g_t}$ are the distinct conjugates of $\theta$ in $G$. Assume also that $$G=\bigcup_{j=1}^tI_G(\theta)g_j,\,\,\textrm{with}\,\,t=[G:I_G(\theta)].$$
Then,
\begin{itemize}
\item[(a)] $\displaystyle Res_{H}^G Ind_H^G \theta=|I_G(\theta)/H|\sum_{j=1}^t\theta^{g_j}.$
\item [(b)] $\displaystyle \langle Ind_H^G\theta,Ind_H^G\theta\rangle=|I_G(\theta)/H|.$ In particular, $Ind_H^G\theta\in Irr(G)$ if and only if $I_G(\theta)=H$.
 \item [(c)]$\displaystyle Res_H^G\chi=e\sum_{j=1}^t \theta^{g_j}$. In particular,
 $$\chi(1)=e t\theta(1)\,\,\,\,\,\textrm{and}\,\,\,\,\,\langle Res_H^G\chi,Res_H^G\chi\rangle=e^2 t.$$
Also, $\displaystyle e^2\leq |I_G(\theta)/H|$ and $e^2t\leq |G/H|.$
\end{itemize}
\end{thm}

In order to ensure the existence of a sequence of irreducible characters of growing degrees, we prove the following corollary which is a consequence of Clifford's Theorem which is given as an exercise in \cite[p.~98]{I1}.

\begin{cor}\label{grandecito}

Let $G$ be a group with a chain of normal subgroups
$$1=H_0\trianglelefteq H_1\trianglelefteq H_2\ldots\trianglelefteq H_n=G$$
such that $H_i/H_{i-1}$ is non abelian for $i=1,\ldots,n$. Then, there exists an irreducible character $\phi$ of $G$, such that $\phi(1)\geq 2^n$.
\end{cor}

\begin{proof}

Let us prove it by induction on $n$. For $n=1$, note that since $H_1\cong H_1/H_0$ is non abelian there exists a character $\psi_1\in Irr(H_1)$ with $\psi_1(1)\geq 2$.

Let us see the case  $n=2$, which illustrates how to proceed in the general case. Let $\psi_2\in Irr(H_2)$ such that  $\psi_1$ is an irreducible constituent of $Res^{H_2}_{H_1}\psi_2$ (take for example $\psi_2\in Ind^{H_2}_{H_1}\psi_1$). From  Theorem \ref{clif}, it follows that

\begin{equation}\label{clic}
Res^{H_2}_{H_1}\psi_2=e\sum_{i=1}^t\psi_{1,i},
\end{equation}
where $\psi_{1,1}=\psi_1,\ldots,\psi_{1,t}$ are the distinct conjugates of $\psi_1$ in $H_2$, $e=\langle Res^{H_2}_{H_1}\psi_2,\psi_1\rangle$ and $t=[H_2:I_{H_2}(\psi_1)].$

If $t=1$, then $H_2=I_{H_2}(\psi_1)$ and $Res^{H_2}_{H_1}\psi_2=e\psi_1$. Hence,
    \begin{itemize}
\item[(a)] $\psi_2(1)=e\psi_1(1)\geq 2$.
\item[(b)] For all $\beta\in Irr(H_2/H_1)$, the character $\psi_2 \beta$ belongs to $Irr(H_2)$ (see \cite[Theorem~19.5]{Hup}). As $H_2/H_1$ is non abelian, we choose $\beta_2\in Irr(H_2/H_1)$ such that $\beta_2(1)\geq 2$. Then, $\phi=\psi_2\beta_2  \in Irr(H_2)$ and $\phi(1)=\psi_2(1)\beta_2(1)\geq 2^2.$
\end{itemize}

 If $t\geq 2$, we just take $\phi=\psi_2$ so, $\phi(1)=et \psi_1(1)\geq 2\psi_1(1)\geq 2^2$.

Now, suppose that it is true for every $m<n$. Then there is $\psi_{m}\in Irr(H_{m})$ such that  $\psi_{m}(1)\geq 2^{m}$. Let us choose $\psi_{m+1}\in Irr(H_{m+1})$ such that $ \psi_{m}$ is an irreducible constituent of  $Res^{H_{m+1}}_{H_{m}}\psi_{m+1}$. As in \eqref{clic},
$$Res^{H_{m+1}}_{H_{m}}\psi_{m+1}=e\sum_{i=1}^t\psi_{m,i},$$
where $\psi_{m,1}=\psi_m,\ldots,\psi_{m,t}$ are the conjugates of $\psi_m$ in $H_{m+1},$ $e=\langle Res^{H_m}_{H_{m+1}}\psi_{m+1},\psi_m\rangle$ and $t=[H_{m+1}:I_{H_{m+1}}(\psi_m)].$

If $t=1$, then

\begin{itemize}
\item[(a)] $\psi_{m+1}(1)\geq 2^{m}$;
\item[(b)] $\psi_{m+1}\beta \in Irr(H_{m+1})$, with $\beta\in Irr(H_{m+1}/H_m)$. Let us choose $\beta$ such that $\beta(1)\geq 2$, then $\phi=\beta\psi_{m+1}\in Irr(H_{m+1})$ and $\phi(1)=\beta(1) \psi_{m+1}(1) \geq 2^{m+1}$.
\end{itemize}
If $t\geq 2$, we take $\phi=\psi_{m+1}$  and $\phi(1)\geq et\psi_m(1)\geq 2^{m+1}$.

\end{proof}

Now we state  the following result which is  crucial for  the proof of Theorem \ref{refe}.

\begin{prop}\label{new}
Let $H$ be a normal subgroup of a finite group $G$. Let $\theta\in Irr(H)$. Then there exists $\rho\in Irr(G)$ such that:
\begin{itemize}
\item [(i)]$\rho(1)\geq [G:H]^{-1/2}\theta(1),$
\item [(ii)] $\langle Ind_H^G \theta, \rho\rangle=a\geq 1$.
\end{itemize}

\end{prop}

\begin{proof}
By the Frobeius reciprocity formula, we get
$$\langle Ind_H^G\theta,Ind_H^G \theta \rangle=\langle \theta,Res_H^G(Ind_H^G\theta)\rangle.$$

Let us recall that
$$ Res_H^G(Ind_H^G \theta) (h)=\sum_{s\in S}\theta_s(h),$$
where $S$ is a system of representative classes of $G/H$ and the character $\theta_s(h)=shs^{-1}$, with $h\in H$.  Because $\theta\in Irr(H)$, the $\theta_s's$ are also irreducible. In particular, it is possible to prove that
$$\langle \theta,Res_H^G(Ind_H^G \theta)\rangle _H\leq \# S=[G:H].$$

\noindent If we write
$$Ind_H^G \theta=\sum_{\rho\in T}a_{\rho}\rho,$$
where $T=\{ \rho\in Irr(G):\langle Ind_H^G \theta, \rho\rangle = a_{\rho}\geq 1\}$ it verify that
$$\langle Ind_H^G \theta,Ind_H^G\theta \rangle=\sum_{\rho\in T}a_{\rho}^2.$$
Thus, we get $\displaystyle\sum_{\rho\in T}a_{\rho}^2\leq [G:H],$ which implies
\begin{itemize}
\item[(a)] $a_{\rho}\leq \sqrt{[G:H]}$ for each $\rho\in T$,
\item[(b)] $\# T\leq [G:H].$
\end{itemize}
Observe that
$$[G:H]\theta(1)=Ind_H^G\theta(1)=\sum_{\rho\in T}a_{\rho}\rho(1).$$
If we take $\rho_0\in T$ of maximal degree, we get
$$[G:H]\theta(1)\leq [G:H]^{3/2}\rho_0(1),$$
so
$$\rho_0(1)\geq [G:H]^{-1/2}\theta(1).$$
\end{proof}

 We say that an irreducible character $\theta$ of $H$ \textit{is extendible to} $G$ if there is an irreducible character $\chi$ of $G$ such that $Res_{H}^{G}\chi=\theta$. The following result gives us a criterion to decide when a character is extendible.

\begin{thm}\label{gaa} \textup{ (Gallagher, \cite[p.~225]{Gal})}
Let $G$ be a finite group with  a normal subgroup $H$  of prime index $q$ in $G$. If $\theta\in Irr(H)$ is invariant in $G$ (i.e. $I_G(\theta)=G$), then $\theta$ is extendible to $G$.
\end{thm}

\begin{lem}\label{con}
 Suppose that $G$ is a finite group with a normal subgroup $H$ such that  $[G:H]=q$, where $q$ is a prime number. If $\theta\in\,Irr(H)$, then the inertia group of $\theta$ is either,

\begin{itemize}

\item[(i)] $I_{G}(\theta)=G$, or

\item[(ii)]  $I_{G}(\theta)=H$.

\end{itemize}
\end{lem}
\begin{proof}

Since $H$ is a normal subgroup of $G$,  $H\subset I_{G}(\theta)$. So if we put $t=[G:I_{G}(\theta)]$, then $t\mid[G:H]$. Therefore $t=1$ or $t=q$.
\end{proof}

\begin{thm}\label{repres}

Under the conditions of  Lemma \ref{con}, let $\chi$ be an irreducible character of $G$. Then, either
\begin{itemize}
\item [(i)] $Res_{H}^{G}\chi=\theta$, for some $\theta\in Irr(H)$ or
\item[(ii)] $\chi=Ind_{H}^{G}\theta$, for some $\theta\in Irr(H)$.
\end{itemize}
\end{thm}

\begin{proof}

Let $\chi\in Irr(G)$ and take $\theta\in Irr(H)$ an irreducible constituent of $Res_{H}^G \chi$. Then $Res_{H}^{G}\chi=e\sum_{i=1}^{t}\theta_i$, where $\theta_i$ are the conjugates of $\theta$ and  ${\langle Res_{H}^{G}\chi,\theta\rangle}=e>0$. Let us fix $\theta=\theta_1$.
If $t=1$, then $I_{G}(\theta)=G$ and by the Theorem \ref{gaa}, $\theta$ is extendible to $G$. On the other hand, we know that \cite[Theorem~19.4]{Hup}

\begin{equation*}
e^2\mid [G:H]=q,
\end{equation*}
so $e^2=1$. Therefore,  $e=1$ and $Res_{H}^{G}\chi=\theta$, so we have (i).

If $t=q$ then $I_{G}(\theta)=H$ and, from Theorem \ref{clif}, $Ind_{H}^{G}\theta\in Irr(G)$. By Frobenius reciprocity,

$$\langle Ind_{H}^{G}\theta,\chi\rangle_{G}=\langle\theta,Res_{H}^{G}\chi\rangle_{H}=\overline{\langle Res_{H}^{G}\chi,\theta\rangle}=\overline{e}=e.$$
Since $Ind_{H}^{G}\theta$ and $\chi$ are irreducible and $\theta$ is a consituent of $Res_{H}^{G}\chi$, it follows that $e=1$ and so $Ind_{H}^{G}\theta=\chi$.

\end{proof}

\section{Estimation for the root Artin conductor of irreducible characters of $G_n$}

 Let $L/M$ be a Galois extension and $\chi$ be the character of a linear representation of Gal($L/M$). The Artin conductor attached to $\chi$ is given by the ideal
\begin{equation*}
f_{\chi}=\prod_{\mathfrak{p}\nmid\infty}\mathfrak{p}^{f_{\chi}(\mathfrak{p})},
\end{equation*}
where $$f_{\chi}(\mathfrak{p})=\frac{1}{|G_0|}\sum_{j\geq 0}(|G_j|\chi(1)-\chi(G_j))$$ and $G_i$ is the $i$-th ramification group of the local extension $L_{\mathfrak{b}}/M_{\mathfrak{p}}$ with $\mathfrak{b}$ a prime over $\mathfrak{p}$ and $\chi(G_j)=\sum_{g\in G_j}\chi(g).$

It is well known that if $L$ is an unramified extension of $M$, then $f_{\chi}$ is the trivial ideal. Then, in order to find a family of irreducible representations with bounded root Artin conductor, let us consider a number field $K$  with infinite $p$-class field tower for some prime $p$. Let $K_n$ be the Hilbert $p$-class field of $K_{n-1}$ with $K_0=K$ and $G_n$=Gal($K_n/{\Q}$).

The main objective of this section is to prove that, under some conditions over $K$ and applying the results of the previous section, there exists an upper bound for the root Artin conductor of the irreducible characters of $G_n$. This bound  depends only on the base field $K$. In addition, we obtain that for each $n>1$ it is possible to find an irreducible character of $G_n$ with  degree increasing with $n$.

\begin{prop}\label{pgroup}
Let $K$ be a Galois extension of ${\Q}$  with infinite $p$-class field tower, for some prime $p$. Suppose that $K$ has a subfield  $\widetilde{k}$  satisfying the following conditions:
\begin{itemize}
\item[(a)] $\widetilde{k}$ is  Galois over ${\Q}$.
\item[(b)] $[\widetilde{k}:{\Q}]=q$, with $q$ a prime number.
\end{itemize}
Let $\chi\in Irr(G_n)$, where $G_n=Gal(K_n/{\Q})$. If $\widetilde{H}_n=Gal(K_n/\widetilde{k})$, then either
\begin{itemize}
\item [(i)] $Res_{\widetilde{H}_n}^{G_n}\chi=\theta$, for some $\theta\in Irr(\widetilde{H}_n)$, or
\item[(ii)] $\chi=Ind_{\widetilde{H}_n}^{G_n}\theta$, for some $\theta\in Irr(\widetilde{H}_n)$ .
\end{itemize}

\end{prop}

\begin{proof}
The proof follows directly from Theorem \ref{repres} with $G=G_n$ and $H=\widetilde{H}_n$.
\end{proof}

\begin{prop}\label{grande}
Let $K$ be a number field with infinite $p$-class field tower for some prime $p$. If $T_n$=Gal($K_n/K$), then
for each $n\geq 1$ there exists  $\phi\in Irr(T_n)$ such that $$\phi(1)> 2^{\frac{n-1}{2}}.$$
\end{prop}
\begin{proof}

Let us consider the following chain of subgroups. If $n$ is even, we take  for $1\leq j\leq \frac{n}{2}$:
\begin{eqnarray*}
H_0&=&\{1\},\\
H_1&=&Gal(K_n/K_{n-2}),\,\,\,\,\,H_1/H_0\cong H_1\\
H_2&=&Gal(K_n/K_{n-4}),\,\,\,\,\,H_2/H_1\cong Gal(K_{n-2}/K_{n-4})\\
&\vdots&\\
H_j&=&Gal(K_n/K_{n-2j}),\,\,\,\,\,H_j/H_{j-1}\cong Gal(K_{n-2(j-1)}/K_{n-2j})\\
&\vdots&\\
H_{\frac{n}{2}}&=&T_n=Gal(K_n/K),\,\,\,\,\,H_{\frac{n}{2}}/H_{\frac{n}{2}-1}\cong Gal(K_2/K).
\end{eqnarray*}
If $l<i-1$ then $K_i/K_l$ is a non abelian group, so by Corollary \ref{grandecito}, there exists $\phi\in Irr(T_n)$ with $\phi(1)\geq 2^{\frac{n}{2}}> 2^{\frac{n-1}{2}}.$

If $n$ is odd, for $j<\frac{n-1}{2}$ we take $H_j$  and  $H_j/H_{j-1}$ as in the  even case. For $j=(n-1)/2 $ we take $H_{\frac{n-1}{2}}$=$T_n$ and $H_{\frac{n-1}{2}}/H_{\frac{n-1}{2}-1}\cong Gal(K_3/K)$. Hence, there exists $\phi\in Irr(G)$ such that $\phi(1)> 2^{\frac{n-1}{2}}.$

\end{proof}

\begin{cor}\label{grande2}
Let $G_n$ be as in Proposition \ref{pgroup}. Then for each $n>1$, there exists $\chi\in Irr(G_n)$ such that $$\chi(1)> 2^{\frac{n-1}{2}}.$$
\end{cor}

\begin{proof}
Note that if $T_n=Gal(K_n/K)$ has an irreducible character $\theta$ with $\theta(1)> 2^{\frac{n-1}{2}}$, then there exists $\chi\in Irr(G)$ with $\chi(1)> 2^{\frac{n-1}{2}}$.  In fact, let $\theta\in Irr(T_n)$ with $\theta(1)> 2^{\frac{n-1}{2}}$ and choose $\chi\in Irr(G_n)$ such that $\theta$ is an irreducible constituent of $Res_{T_{n}}^{G_n} \chi$. By Theorem \ref{clif}, $\chi(1)=et\theta(1)$, where $e=\langle Res_{T_{n}}^{G_n} \chi,\theta\rangle$ and $t=[G_n:I_g(\theta)]$. As $e,t\geq 1$, then $\chi(1)\geq \theta(1)> 2^{\frac{n-1}{2}}$.

\end{proof}

Now, we  obtain upper bounds for the root Artin conductor of irreducible characters of $G_n$.

\begin{thm}\label{bound}
Assume $G_n$ as in Proposition \ref{pgroup} and $\chi\in Irr(G_n)$.
\begin{itemize}
\item[(i)] If  $Res_{\widetilde{H}_n}^{G_n}\chi=\theta$, for some $\theta\in Irr(\widetilde{H}_n)$ then
$$f_{\chi}^{1/\chi(1)}\leq |D_{\widetilde{k}/{\Q}}|N_{\widetilde{k}/{\Q}}(f_{\theta})^{1/{\theta(1)}}.$$
\item[(ii)] If $\chi=Ind_{\widetilde{H}_n}^{G_n}\theta$, for some $\theta\in Irr(\widetilde{H}_n)$ then
$$f_{\chi}^{1/\chi(1)}=|D_{\widetilde{k}/{\Q}}|^{1/q}N_{\widetilde{k}/{\Q}}(f_{\theta})^{1/{q\theta(1)}}.$$
\end{itemize}

\end{thm}

\begin{proof}
In the first case, we have $\chi(1)=\theta(1)$ and
$$Ind_{\widetilde{H}_n}^{G_n}\theta=\sum_{i=1}^q\psi_i(1)\cdot\chi\psi_i,$$
where $Irr(G_n/\widetilde{H}_n)=\{\psi_1,\psi_2,\ldots,\psi_q\}$ (see  \cite[Theorem~19.5]{Hup}).
 Since $G_n/\widetilde{H}_n$ is isomorphic to the abelian group ${\Z}/q{\Z}$, it follows that $Ind_{\widetilde{H}_n}^{G_n}\theta=\sum_{i=1}^q\chi\psi_i$.
The Artin conductor of this induced character is, on the one hand,
$$f_{Ind_{\widetilde{H}_n}^{G_n}\theta}=|D_{\widetilde{k}/{\Q}}|^{\theta(1)}N_{\widetilde{k}/{\Q}}(f_{\theta}),$$
where the ideal $f_{\theta}$ is the Artin conductor of $\theta$.
On the other hand, assuming that $\psi_1$ is the trivial character,
$$f_{Ind_{\widetilde{H}_n}^{G_n}\theta}=f_{\sum_{i=1}^q\chi\psi_i}=f_{\chi}\cdot\prod_{i=2}^qf_{\chi\psi_i}.$$

\noindent Now, combining these expressions we get
$$f_{\chi}= |D_{\widetilde{k}/{\Q}}|^{\theta(1)}N_{\widetilde{k}/{\Q}}(f_{\theta})\cdot\left(\prod_{i=2}^qf_{\chi\psi_i}\right)^{-1},$$
so
\begin{equation*}
f_{\chi}^{1/\chi(1)}\leq |D_{\widetilde{k}/{\Q}}|N_{\widetilde{k}/{\Q}}(f_{\theta})^{1/\theta(1)}.
\end{equation*}

In the second case,
$$\chi(1)=[G_n:\widetilde{H}_n]\theta(1)=q\theta(1)$$
and we can see that the root Artin conductor of $\chi$ is given by the expression

\begin{equation*}
f_{\chi}^{1/\chi(1)}=|D_{\widetilde{k}/{\Q}}|^{1/q}N_{\widetilde{k}/{\Q}}(f_{\theta})^{1/q\theta(1)}.
\end{equation*}

\end{proof}

\noindent In order to obtain a  bound for the root Artin conductors, we need the following result.

\begin{lem}\label{ramifi}
  Assume $K_n$ and $K$ as in the Proposition \ref{pgroup}. Let $\mathfrak{p}$ be a prime in $\widetilde{k}$, with $\mathfrak{b}$ and $\mathfrak{q}$ primes over $\mathfrak{p}$ in $K_n$ and $K$ respectively. Let $G_i(K_{n,\mathfrak{b}}/\widetilde{k}_{\mathfrak{p}})$  and $G_i(K_{\mathfrak{q}}/\widetilde{k_{\mathfrak{p}}})$ be the $i$-th ramification groups of the local extensions $K_{n,\mathfrak{b}}/\widetilde{k}_{\mathfrak{p}}$ and $K_{\mathfrak{q}}/\widetilde{k_{\mathfrak{p}}}$. Then, for $i\geq 0$

\begin{itemize}

\item[(a)] $G_i(K_{n,\mathfrak{b}}/K_{\mathfrak{q}})=G_i(K_{n,\mathfrak{b}}/\widetilde{k_{\mathfrak{p}}})\cap G(K_{n,\mathfrak{b}}/K_{\mathfrak{q}})=\{1\}$
\item[(b)] $|G_i(K_{n,\mathfrak{b}}/\widetilde{k_{\mathfrak{p}}})|=|G_i(K_{\mathfrak{q}}/\widetilde{k_{\mathfrak{p}}})|.$
\end{itemize}
\end{lem}

\noindent The proof of this lemma follows directly from properties of higher ramification groups (see for example \cite[p.177-180]{Ne}) and by the fact that $K_n/K$ is an unramified extension.

\begin{cor}
There is an infinite sequence $\{\chi_n\}_{n\in{\N}}$ of irreducible Artin characters with $\chi_n(1)\to\infty$ and with
$$f_{\chi_n}^{1/{\chi_n(1)}}\leq C,$$
where  $C>0$ is an effective computable constant.
\end{cor}

\begin{proof}
By the  Corollary \ref{grande2} and Theorem \ref{bound}, we know that for each $n$ there is an irreducible character $\chi_n$ of $G_n$ with $\chi_n(1)\to\infty$ and  $$f_{\chi_n}^{1/\chi_n(1)}\leq |D_{\widetilde{k}/{\Q}}|N_{\widetilde{k}/{\Q}}(f_{\theta})^{1/\theta(1)},$$
for some $\theta\in Irr(\widetilde{H}_n)$.
By the properties of the higher ramification groups stated in Lemma \ref{ramifi} and considering that the primes ramifying in $K$ are the only ones  that appears in $N_{\widetilde{k}/{\Q}}(f_{\theta})$, it is possible to find a constant $T>0$ depending only on the base field $K$, such that $N_{\widetilde{k}/{\Q}}(f_{\theta})\leq T^{\theta(1)}$.
Hence,
$$f_{\chi_n}^{1/\chi_n(1)}\leq |D_{\widetilde{k}/{\Q}}|T:=C.$$
\end{proof}

\begin{remark}
As the referee pointed out, it is possible to avoid the hypothesis about the degree of $\widetilde{k}/{\Q}$ and obtain the same type of bounds for the asymptotic behavior of $f_{\chi}^{1/\chi(1)}$. This is accomplished in Theorem \ref{refe} below.

\begin{thm}\label{refe}
Let $K$ be a Galois extension of ${\Q}$ with infinite $p$-class field tower. Let $m=[K:{\Q}]$. Then there exists an infinite sequence
$\{\chi_n \}_{n\in{\N}}$ of irreducible Artin characters such that $\chi_n(1)\to\infty$ and
$$f_{\chi_n}^{1/\chi_n(1)}\leq |D_{K/{\Q}}|^{m^{1/2}}.$$
\end{thm}

\begin{proof}
Let $G_n=Gal(K_n/{\Q})$ and $T_n=Gal(K_n/K)$.  We can choose $\theta_n\in Irr(T_n)$ as in Proposition \ref{grande}. Then, by the Proposition \ref{new},  there exists  $\chi_n\in Irr(G_n)$  such that  $\langle Ind_{H_n}^{G_n}\theta_n,\chi_n\rangle =a\geq 1$ and with $\chi_n(1)>m^{-1/2}\theta_n(1)$, so

\begin{equation*}\label{dess}
\chi_n(1)>m^{-1/2}2^{\frac{n-1}{2}}.
\end{equation*}

\noindent Hence, by the properties of the Artin conductor we get
$$f_{\chi_n}^a\leq f_{Ind_{H_n}^{G_n}\theta_n}=|D_{K/{\Q}}|^{\theta_n(1)},$$
and therefore,
$$f_{\chi_n}^{1/\chi_n(1)}\leq f_{\chi_n}^{a/\chi_n(1)}\leq |D_{K/{\Q}}|^{m^{1/2}}.$$

\end{proof}
\end{remark}

\section{Number fields with infinite $2$-class field tower}

Golod and Shafarevich \cite{GS} proved that a number field $K$ has an infinite $p$-class field tower if  the $p$-rank of the class group of $K$ is large enough. In this case,
$$\alpha(r_1,r_2)\leq |D_K|^{1/[K:{\Q}]},$$
where $D_K$ is the discriminant of $K$.

In addition, Martinet has constructed a number field with infinite Hilbert class field towers and lowest known root discriminant and  proved that
\begin{equation*}
\alpha(0,1) < 93\,\,\textrm{and}\,\,\alpha(1,0) < 1059.
\end{equation*}
In particular, he found that $K=\Q(\zeta_{11}+\zeta_{11}^{-1},\sqrt{2},\sqrt{-23})$ has infinite 2-class field tower.
Since $\widetilde{k}=\Q(\zeta_{11}+\zeta_{11}^{-1})$ is a subfield of $K$ of degree 5 over $\Q$,  $K$ satisfies the conditions of the Theorem \ref{bound}.
The discriminant of $\widetilde{k}$ is
$$|D_{\widetilde{k}/{\Q}}|=14641=11^4$$
and the only rational primes that ramify in $K$ are 2, 11 and 23.
Using PARI/GP \cite{PARI2}, we can estimates the sizes of the higher ramification groups. Thus, we get the upper bound
\begin{equation*}
N_{\widetilde{k}/{\Q}}(f_{\theta})\leq (2^{15}23)^{\theta(1)}.
\end{equation*}
With this estimation,  it  follows the explicit result:

\begin{cor}\label{gene}
For each $n\geq 1$, there exists a irreducible character $\chi_n$ such that $\chi_n(1)\to\infty$  and
$$f_{\chi_n}^{1/\chi_n(1)}\leq C,\,\,\textrm{where}\,\,C \leq 11^4\cdot 2^{15}\cdot 23.$$
\end{cor}
An open problem now is to improve the constant $C$.

\vspace{0.3cm}

\noindent \textbf{Acknowledgements}. This research was partially supported by the ``Red Iberoamericana de Teor\'ia de N\'umeros'', ECOS-CONICYT grant 170022 ``Explicit Arithmetic Geometry'' and IdEx-Universit\'e de Bordeaux (credits to Benjamin Matschke). The author would like to thank Eduardo Friedman for suggesting this problem and  his helpful advice.
The author also thanks Bill Allombert, Yuri Bilu, Mariela Carvacho, Milton Espinoza and Andrea Vera for useful discussions. Finally, the author thanks the referees for the comments and in particular for suggesting  Proposition 3 and Theorem 14. While working  on this project, the author visited the ICMAT Madrid and IMB Bordeaux.

\bibliographystyle{amsplain}
\bibliography{refer}

\providecommand{\bysame}{\leavevmode\hbox to3em{\hrulefill}\thinspace}
\providecommand{\MR}{\relax\ifhmode\unskip\space\fi MR }
\providecommand{\MRhref}[2]{%
  \href{http://www.ams.org/mathscinet-getitem?mr=#1}{#2}
}
\providecommand{\href}[2]{#2}
\begin{thebibliography}{1}

\bibitem{Gal}
P.~X. Gallagher, \emph{Group characters and normal {H}all subgroups}, Nagoya
  Math. J. \textbf{21} (1962), 223--230.

\bibitem{GS}
E.~S. Golod and I.~R. {\v{S}}afarevi{\v{c}}, \emph{On the class field tower},
  Izv. Akad. Nauk SSSR Ser. Mat. \textbf{28} (1964), 261--272.

\bibitem{Hup}
B.\ Huppert, \emph{Character theory of finite groups}, De Gruyter, Berlin{-}New
  York, 1998.

\bibitem{I1}
I.~Martin Isaac, \emph{Character theory of finite groups}, AMS Chelsea
  Publishing, 2006.

\bibitem{A1}
A.\ Pizarro~{-} Madariaga, \emph{Lower bounds for the {A}rtin conductor},
  Mathematics of Computation \textbf{80} (2011), no.~273, 539--561.

\bibitem{Ma1}
J.\ Martinet, \emph{Tours de corps de classes et estimations de discriminants},
  Inventiones Mathematicae \textbf{44} (1978), 65{--}73.

\bibitem{Ne}
J.~Neukirch, \emph{Algebraic number theory}, Springer-Verlag, Berlin, 1999.

\bibitem{O2}
A.M.\ Odlyzko, \emph{On conductors and discriminants}, Proc. Sympos., Univ.
  Durham \textbf{1} (1977), 377--407.

\bibitem{PARI2}
{The PARI~Group}, Bordeaux, \emph{{PARI/GP version {\tt 2.7.1}}}, 2014,
  available from \url{http://pari.math.u-bordeaux.fr/}.

\end{thebibliography}
\end{document}